\documentclass[a4paper, 11pt]{article}
\usepackage[pdftex]{hyperref} 
\usepackage{srcltx}
\usepackage{indentfirst}
\usepackage{graphicx}
\usepackage{overpic}
\usepackage{multirow}
\usepackage[hang]{subfigure}
\usepackage{amsmath, amssymb, amsthm}
\usepackage{color}
\usepackage[mathscr]{euscript}
\usepackage{mathrsfs} 

\newtheorem{theorem}{Theorem}
\newtheorem{lemma}{Lemma}

\newcommand\bbR{\mathbb{R}}

\newcommand\bbE{\mathbb{E}}
\newcommand\bbN{\mathbb{N}}

\newcommand\dd{\,\mathrm{d}}

\newcommand\pd[2]{\dfrac{\partial {#1}}{\partial {#2}}}
\newcommand\opd[2]{\dfrac{\dd {#1}}{\dd {#2}}}

\newcommand \ri{\mathrm{i}}

\newcommand \Vw{\mathcal{V}}  
\newcommand\Rv{ {\mathbb{R}_{v} } }

\newcommand\BigNorm[1]{ \left \| #1 \right \| } 

\newcommand \FTFF[3]{\mathcal{F}_{ {#1}\rightarrow{#2}  } \left( #3 \right) }

\newcommand \IFTFF[3]{\mathcal{F}^{-1}_{ {#1}\rightarrow{#2}  } \left( #3 \right) }
\newcommand \bounded[1]{ \mathcal{L}\left( #1 \right) }

\numberwithin{equation}{section}

\newcommand\assign{:=}

\graphicspath{{images/}}

{\theoremstyle{remark} }

\title{Parity-decomposition and Moment Analysis for Stationary Wigner
Equation with Inflow Boundary Conditions}

\author{Ruo Li\thanks{HEDPS \& CAPT, LMAM \& School of Mathematical
    Sciences, Peking University, Beijing, China, email: {\tt
  rli@math.pku.edu.cn}.}, ~~ Tiao Lu\thanks{CAPT, HEDPS, LMAM,
IFSA Collaborative Innovation Center of MoE, \&
    School of Mathematical Sciences, Peking University, Beijing,
    China, email: {\tt tlu@math.pku.edu.cn}.},
    ~~ Zhangpeng Sun \thanks{School of Mathematical Sciences, 
	  Peking University, Beijing, China, 
	  email: {\tt sunzhangpeng@pku.edu.cn}.}
}

\begin{document}

\maketitle
\begin{abstract}
We study the stationary Wigner equation on a bounded, one-dimensional
spatial domain with inflow boundary conditions by using the parity
decomposition in (Barletti and Zweifel, Trans. Theory Stat.
Phys., 507--520, 2001). The decomposition reduces the
half-range, two-point boundary value problem into two decoupled
initial value problems of the even part and the odd part. Without
using a cutoff approximation around zero velocity, we prove that the
initial value problem for the even part is well-posed.
For the odd part, we prove the uniqueness of the solution in the odd
$L^2$-space by analyzing the moment system. An example is provided to
show that how to use the analysis to obtain the solution of the
stationary Wigner equation with inflow boundary conditions.
\end{abstract}

  \vspace*{4mm}


\section{Introduction} \label{sec:intro} 

As the size of electronic devices approaches the nanometer scale,
quantum effects, such as tunneling, have to be considered in study of
the device properties. As a result, quantum models, including
the Schr\"odinger equation, non-equilibrium Green function methods and 
the Wigner equation, have attracted increasing attentions. In these
models, the Wigner equation has some advantages over other quantum
models \cite{Ferry_transport}. One advantage is that 
the inflow boundary conditions for the Boltzmann equation can be
extended to the Wigner equation since the latter can be formulated as
the former with a quantum correction term. Especially, the stationary
Wigner equation with inflow boundary conditions is often adopted in
numerical simulation of nanoscale devices. Starting from
\cite{Frensley1987}, simulations of nanoscale  devices using such a
model have provided a lot of encouraging numerical 
results \cite{Jensen1990, Tsuchiya1991, 
Shih1994, GeKo05, Querlioz2006, ShLu09, Kelley2010, LiLuSun2014}. 

Even though, rigorous mathematical theory on the well-posedness of the
stationary Wigner equation with inflow boundary conditions is
still an open problem, even for one dimensional case
\cite{Barletti2003}. We note that there are some results in a
semi-discrete version, such as \cite{ALZ00}.
In \cite{LiWigner2014}, the semi-discrete version of the Wigner
equation is related to the truncated Wigner equation proposed in
\cite{JiangLuCai2013} using the Shannon sampling theory.  
The truncation length is called a coherence length in many papers,
e.g., \cite{Jiang2011}. 

The well-posedness of the stationary continuous Wigner equation
with inflow boundary conditions is still a temptatious problem
mathematically. We adopt the parity decomposition technique to study
the stationary Wigner equation with inflow boundary conditions, which
has been proposed in \cite{Zweifel2001}. As in
\cite{Zweifel2001}, the Wigner equation with inflow boundary
conditions is a linear boundary value problem (BVP), and the even and
odd parts are decoupled and each part is a solution of the Wigner BVP
with corresponding boundary conditions.
It was pointed out in \cite{Zweifel2001} that: ``{\it It is worth to
  remark   that we do not obtain a well-posedness result for the full
  problem   $($that is, with $v$ in place of $\eta_{\epsilon}(v))$ by
  simply   letting $\epsilon$ go to $0$. The analysis of the full
  problem must   be carried out by means of more sophisticated
  techniques than those
  employed here. In particular, $B(x)$ could be studied as an
  unbounded linear evolution operator in a suitable space, with the
  initial datum $((f_{b,e}^+(-l/2), f_{b,o}^+ (-l/2)))$ restricted to
  the appropriated domain.}'' 
In \cite{Zweifel2001}, a small interval centered at $v=0$ is removed
to obtain the well-posedness result. Here we will try to clarify the
questions put forward therein. Without a cutoff approximation around
$v=0$, we prove that the pseudo-differential operator $B(x)$ (defined
in \eqref{BxOperator}) is a bounded linear operator on the even
$L^2$-space, $L^2_e(\bbR_v)$ (defined in \eqref{DefL2e}), only if the
potential function is regular enough. Thus, we can obtain the
well-posedness of the even part directly. Generally, $B(x)$ is  on
longer a bounded operator on the odd $L^2$-space,
$L^2_o(\bbR_v)$(defined in \eqref{DefL2o}). However, we prove the
uniqueness of the solution of the odd part in $L^2_o(\bbR_v)$ by
analyzing its moment system. 

The rest of this paper is organized as follows. In Section 2, we
present the governing equations and in Section 3, we introduce the
parity decomposition. The equation with inflow boundary conditions is
discussed in Section 4 and then a short conclusion closes 
the main text.


\section{Wigner Equation}
We consider the stationary, linear Wigner equation of the
form \cite{Wigner1932}
\begin{equation} \label{WignerBounded}
  v \pd{ f(x,v)}{x} - \Theta f(x,v) = 0, \quad x \in [-l/2,l/2], v \in
  \Rv.    
\end{equation}
For convenience, we have set the reduced Planck constant $\hbar$, the
electron charge $e$ and the effective mass of electron $m$ to be equal
to unity. Here $\Theta$ is an anti-symmetric pseudo-differential
operator. Precisely,  
\begin{equation} \label{Theta}
  (\Theta f) (x,v) = \ri \IFTFF{y}{v}{ 
  [V(x+y/2)-V(x-y/2)] \hat{f}(x,y) }, 
\end{equation} 
where $V: \bbR \rightarrow \bbR$ is the potential. Using the convolution
theorem of the Fourier transform, we have 
\begin{equation}
  (\Theta f)(x,v) = \int_{\Rv}\Vw(x,v-v') f(x,v') \dd v',  
\end{equation}
where $\Vw(x,v)$ is defined in \eqref{VwDef}.
We use $\hat{f}(x,y) = \FTFF{v}{y}{f(x,v)}$  to denote the Fourier 
transform of $f(x,v)$, and 
\begin{equation}
  \FTFF{v}{y}{f(x,v)}=
  \int_{\Rv} f(x,v) \exp ( - \ri v y ) \dd v . 
\end{equation}
Correspondingly, the inverse Fourier transform of $\hat{f}(x,y)$ is 
defined as 
\begin{equation}
  \IFTFF{y}{v}{\hat{f}(x,y)}= \frac{1}{2\pi}\int_{\bbR}\hat{f}(x,y)
  \exp(\ri v y ) \dd y .
\end{equation}

The derivation of the Wigner equation from the Schr\"odinger equation
can be found in many references, e.g., \cite{Hillery1984, 
JiangLuCai2013, Ruo2012}.  Here, we only describe the Wigner-Weyl 
transform simply for completeness of the paper. 
A quantum system is described by the Schr\"odinger equation 
\begin{equation} \label{sch}
  \left( -\frac{1}{2} \opd{^2}{x^2} 
  + V(x) \right) \psi_n(x) = \bbE_n \psi_n(x),  
\end{equation}
where the eigen-function $\psi_n(x)$ is called a pure state and the
associated eigenvalue $\bbE_n \in \bbR$ is called an eigen energy. The
density matrix $\rho(x,x')$ for a mixed state is defined as
\begin{equation}
  \rho(x,x') = \sum_n P_n \psi^{\ast}_n(x)\psi_n(x') 
\end{equation}
where $P_n$ is the probability of the electron occupying the state
$\psi_n$, and $\sum_n P_n = 1$. The stationary Liouville-von Neumann 
equation is then derived from \eqref{sch} as
\begin{equation} \label{von}
  \left[ -\frac{1}{2} \left(\pd{^2}{x^2} - \pd{^2}{ {x'}^2 } 
  \right) + V(x)-V(x') \right] \rho(x,x') = 0  
\end{equation}
Introducing the quasi-probability distribution function 
\begin{equation} \label{fxvDef}
  f(x,v) = \IFTFF{y}{v}{\rho(x+y/2,x-y/2)}, 
\end{equation}
and applying a change of variables and the inverse Fourier transform of
\eqref{von}, we derive the Wigner equation governing $f(x,v)$ as
\begin{equation} \label{Wigner}
  v \pd{f(x,v)}{x} - \Theta f(x,v) = 0, \quad x\in \bbR, v \in \Rv,  
\end{equation}
where $\Theta f$ is defined in \eqref{Theta}. 

We need to specify some conditions to ensure that \eqref{Wigner} or
\eqref{WignerBounded} has a unique solution. Before the discussion on
these conditions, let us to examine the properties of the stationary
Wigner equation at first.


\section{Parity Decomposition}
Following \cite{Zweifel2001}, we take the point of view that the 
function $f(x,\cdot):\Rv \rightarrow \bbR$ belongs to the Hilbert space
$L^2(\Rv)$ for every fixed $x\in \bbR$, equipped with the
norm
\begin{equation}
\BigNorm{f(x,\cdot)}_{L^2(\Rv)} =\left( 
   \int_{\Rv} |f(x,v)|^2 \dd v \right)^{1/2}. 
\end{equation}
Thus, $f$ is regarded as a vector-valued function from $\bbR$ to $L^2(\Rv)$.
To emphasize this, we will often use the notation $[f(x)](v)$ instead
of $f(x,v)$ later on. With such notations, the Wigner equation
\eqref{Wigner} is recast into an equation for the unknown function
$f:\bbR \rightarrow L^2(\Rv)$ of the following form
\begin{equation} \label{WignerVec}
\opd{}{x}f(x)  - B(x) f (x) = 0, \quad x \in \bbR. 
\end{equation}
Here, $\opd{}{x}$ is a differential operator on $C^{1}(\bbR; L^2(\Rv))$ and the
operator $B(x)$ is defined by
\begin{equation} \label{BxOperator}
[B(x) f(x)](v)  = \frac{\ri }{v}  \IFTFF{y}{v}{
   D_V(x,y) [\widehat{f(x)}] (y) },  
\end{equation}
where 
\[
D_V(x,y) = V(x+y/2)-V(x-y/2), \quad
[ \widehat{f(x)} ](y) = \FTFF{v}{y}{ [f(x)](v)}. 
\]

We decompose the space $L^2(\Rv)$ into the direct sum $L^2_e(\Rv)
\bigoplus L^2_o(\Rv)$ where $L^2_e(\Rv)$ and $L^2_o(\Rv)$ are the subspaces of $L^2(\Rv)$
defined by 
\begin{equation}\label{DefL2e}
L^2_e(\Rv) = \{ u(v) \in L^2(\Rv) : u(v) = u(-v)\},
\end{equation}
\begin{equation} \label{DefL2o}
L^2_o(\Rv) = \{ u(v) \in L^2(\Rv) : u(v) = -u(-v)\}. 
\end{equation}
Clearly, for $\forall f\in L^2(\Rv)$, $f$ is uniquely decomposed into
the sum
\[
f = f_e + f_o,
\]
where $f_e$ and $f_o$ are its even part and odd part, respectively,
i.e.,   
\[
f_e(v) = [P_e f](v) \assign \frac{1}{2} \left( f(v) + f(-v) \right),
  \quad 
f_o(v) = [P_o f](v) \assign \frac{1}{2} \left( f(v) - f(-v) \right),
\]
where $P_e : L^2(\Rv)\rightarrow L^2_e(\Rv)$ and $P_o: L^2(\Rv)
\rightarrow L^2_o(\Rv)$ are two projection operators defined in
the above equations.

As pointed out in \cite{Zweifel2001}, $B(x)$ is an even operator which
preserves the parity. Therefore, the subspace $L^2_e(\Rv)$ and
$L^2_o(\Rv)$ are closed with exertion of $B(x)$. This fact is formally
expressed by the commutation relations
\[
P_e B(x) = B(x) P_e, \quad 
P_o B(x) = B(x) P_o, \quad \forall x \in \bbR.
\]
By applying the operators $P_e$ and $P_o$ on both sides of
\eqref{WignerVec}, we immediately split \eqref{WignerVec} into two
identical decoupled equations for the even part and the odd part of
the unknown $f(x)$:
\begin{equation} \label{WignerVecEven}
\opd{}{x} f_e(x) - B(x) f_e(x) = 0,  \quad x\in \bbR, 
\end{equation} 
\begin{equation} \label{WignerVecOdd} 
\opd{}{x} f_o(x) - B(x) f_o(x) = 0, \quad x \in \bbR.  
\end{equation}
The equations \eqref{WignerVec}, \eqref{WignerVecEven} and
\eqref{WignerVecOdd} are linear ordinary differential equations. This
brings us to consider initial value problems at first.

We consider the following initial value problem (IVP)
\begin{equation}
  \opd{}{x}f(x)  - B(x) f (x) = 0, \quad x \in \bbR,
  \tag*{(\ref{WignerVec})}
\end{equation}
with the initial condition  
\begin{equation} \label{WignerVecIV}
  f(-l/2) = f_b \in L^2(\Rv).
\end{equation}
The initial value $f_b$ can be uniquely decomposed into the sum
\[
  f_b = f_{b,e} + f_{b,o},
\]
where $f_{b,e} = P_e f_b$ and $f_{b,o} = P_o f_b$. By the even
property of the operator $B(x)$, it is easy to verify that if $f(x)$
is the solution of the IVP \eqref{WignerVec}+\eqref{WignerVecIV},
then $f_e = P_e f$ is the solution of the IVP \eqref{WignerVecEven}
with the initial condition
\begin{equation} \label{WignerVecEvenIV}
  f_e(-l/2) = f_{b,e}, 
\end{equation}
and $f_o = P_o f$ is the solution of the IVP \eqref{WignerVecOdd} with
the initial condition 
\begin{equation} \label{WignerVecOddIV}
  f_o(-l/2) = f_{b,o}.  
\end{equation}

In order to study the IVP \eqref{WignerVec}+\eqref{WignerVecIV},
we need only to analyze the IVP 
\eqref{WignerVecEven}+\eqref{WignerVecEvenIV} and the IVP 
\eqref{WignerVecOdd}+\eqref{WignerVecOddIV}, respectively.

\subsection{The even part}
For the even part of the Wigner equation, i.e., the IVP
\eqref{WignerVecEven}+\eqref{WignerVecEvenIV}, we rewrite it as
\begin{equation} \label{WVE1}
  \opd{f(x)}{x} -B(x) f(x) = 0,
\end{equation}
with the initial condition 
\begin{equation} \label{WVEI1}
  f(-l/2) = f_b \in L^2_e(\Rv).
\end{equation}

Below we prove that under some assumptions, there exists a unique
solution $f(x)\in L^2_e(\Rv)$ of the IVP \eqref{WVE1}-\eqref{WVEI1}. 
As a preliminary step to prove the result, we give a
lemma to declare that $B(x)$ is a bounded linear operator on
$L^2_e(\Rv)$.  

\begin{lemma}  \label{BoundedOperator}
  Let 
  \begin{equation} \label{VwDef}
    \Vw(x,v) = \ri \IFTFF{y}{v}{ V(x+y/2)-V(x-y/2) } .  
  \end{equation}
  Assuming $\Vw \left( x, v \right) \in H^1 \left( \Rv \right)$, $B(x) :
  L_e^2 \left( \Rv \right) \rightarrow L_e^2 ( \Rv )$ defined in
  \eqref{BxOperator} can be rewritten into
  \[
    [ B(x)f ](v) = \frac{1}{v} \Vw\ast f(x,v).  
  \]
  Then  $B(x)$ is a bounded linear operator on $L^2_e(\Rv)$.
\end{lemma}
\begin{proof}
  By the definition of $\Vw$ in \eqref{VwDef}, we have $\Vw(x) \in
  L^2_o(v)$. Thus, for $\forall f(x) \in L^2_e(\Rv)$, we have
  \begin{equation} \label{VwZero}
    \int _{\Rv} \Vw(x,v) f(x,v) \dd v = 0. 
  \end{equation}
  We introduce an linear operator $A(x):L^2(\Rv) \rightarrow
  L^2(\Rv)$   
  \begin{equation} \label{OperatorA}
  [ A(x) f ] \left(v \right) =  \int_{\mathbb{R}}
  \frac{\Vw \left( x, v - v' \right) - \Vw \left( x, 0 - v' \right)}{v}
  f \left( x, v' \right) d v'. 
\end{equation}
  By \eqref{VwZero}, it is concluded that $A(x) = B(x)$ on
  $L^2_e(\Rv)$. 
  We will prove that $A(x)$ is a bounded linear operator on
  $L^2(\Rv)$ by estimating it on regions $\left| v \right|
  > 1$ and region $\left| v \right| \leq 1$, respectively.

  First, we consider the part with $\left| v \right| > 1$.  Using $\Vw
  \left( x, v \right) \in L^2 \left( \mathbb{R}_v \right)$ and the
  Young's inequality, we have
  \begin{equation} \label{eq6}
    \BigNorm{\Theta  f( x, v )}_{L^{\infty}} 
    =\|\Vw \left( x, v \right) \ast f \left( x, v
    \right) \|_{L^{\infty} \left( \mathbb{R}_v \right)} \leq \|\Vw
    \left( x, \cdot \right) \|_{L^2} \|f \left( x, \cdot \right) \|_{L^2
      \left( \mathbb{R}_v \right)}. 
  \end{equation} 
  By the Cauchy-Schwartz inequality, we then have  
  \begin{equation} \label{eq7}
    \left|
      \int_{\mathbb{R}_{v'}} \Vw \left( x, 0 - v' \right) f \left( x,
        v' \right) \dd v' \right| \leq \|\Vw \left( x, \cdot \right)
    \|_{L^2}^{} \|f \left( x, \cdot \right) \|_{L^2 \left(
        \mathbb{R}_{v'} \right)} .  
  \end{equation}
  It is obtained directly from \eqref{eq6} and \eqref{eq7} that
  \begin{align}
    \int_{\left| v \right| > 1} \left| [A(x) f] \left(
        v \right) \right|^2 \dd v & =
    2 \int_{\left| v \right| > 1} \left| \frac{\Theta 
        f \left( x, v \right)}{v} \right|^2 \dd v + 2 \int_{\left| v
      \right| > 1} \frac{\|\Vw \left( x, \cdot \right) \|_{L^2}^2 \|f
      \left( x, \cdot \right) \|_{L^2 \left( \mathbb{R}_{v'}
        \right)}^2}{v^2} \dd v  \notag \\ 
    & \leq 8\|\Vw
    \left( x, \cdot \right) \|^2_{L^2} \|f \left( x, \cdot
    \right) \|^2_{L^2} .  \label{eq8}
  \end{align}

  Then, we consider the part with $\left| v \right| \leq
  1$. According to the Cauchy-Schwartz inequality again, we have
  \begin{align*}
    \left| [A(x)  f] \left( v \right) \right| & \leq
    \int_{\mathbb{R}} \left| \frac{\Vw \left( x, v - v'
	\right) - \Vw \left( x, 0 - v' \right)}{v} \right| \left| f \left(
        x, v' \right) \right| \dd v' \\
    & \leq \left \| \frac{\Vw \left( x, v - v' \right) - \Vw \left( x,
          0 - v' \right)}{v} \right\|_{L^2 \left( \mathbb{R}_{v'} \right)} \|f
    \left( x, \cdot \right) \|_{L^2}, \quad v \in \left[ - 1, 1 \right] 
  \end{align*}
  By using Theorem 3 in Chapter 5 of \cite{Evans2010}, we have
  \[
  \BigNorm{ \frac{\Vw \left( x, v_{} - v'
      \right) - \Vw \left( x, - v' \right)}{v} }_{L^2 \left(
      \mathbb{R}_{v'} \right)} \leq 
  \BigNorm{ \partial_{v'} \Vw \left( x,
      v' \right) }_{L^2 \left( \mathbb{R}_{v'} \right)} .
  \]
This fact, together with the Cauchy-Schwartz inequality,
  gives us the following estimate on the velocity interval $[-1,1]$
  that
  \begin{equation} \label{eq9}
    \int_{\left| v \right|
      \leq 1} \left| [A(x) f] \left( v \right) \right|^2 \dd v \leq 
    \|f \left( x, \cdot \right) \|^2_{L^2}   \BigNorm{ \partial_{v}
\Vw \left( x,
      v \right) }^2_{L^2 \left( \mathbb{R}_{v} \right)} 
  \end{equation} 
  Collecting \eqref{eq8} and \eqref{eq9} together results in
  \[
  \| [A(x) f] \left( v \right) \|^2_2
  \leq C\|f \left( x, \cdot \right) \|^{^2}_{L^2} 
  \]
  where 
  \[ C = 8 \BigNorm{\Vw \left( x, \cdot
    \right) } ^{^2}_{H^1} . 
  \]

  We have proved that $A(x)$ is a bounded operator on $L^2(\Rv)$.
  When it is restricted on the subspace $L^2_e(\Rv)$, we have $A(x) =
  B(x)$. This completes the proof that $B(x)$ is a linear bounded
  operator on $L^2_e(\Rv)$.
\end{proof}

By Lemma \ref{BoundedOperator}, we immediately have 

\begin{theorem} \label{EvenWell}
Let $\bounded{L^2_e(\Rv)}$ denote the space of bounded linear operators
on $L^2_e(\Rv)$. If the assumptions for Lemma \ref{BoundedOperator} hold, then 
one has that
\begin{enumerate}
\item[(a)] If $B(x) \in L^1( (-l/2,l/2), \bounded{L^2_e(\Rv)})$ , then
  the IVP \eqref{WVE1}-\eqref{WVEI1} has a unique mild solution $f
  \in W^{1,1}\left( (-l/2,l/2), L^2_e(\Rv) \right)$.
\item[(b)] If $B(x)$ is strongly continuous in $x$ on $[-l/2,l/2]$ and
  uniformly bounded in the norm of $\bounded{L_e^2(\Rv)}$ on
  $[-l/2,l/2]$, then the solution $f$ is a classical solution, i.e.,
  $f \in C^1\left( [-l/2,l/2], L^2_e(\Rv) \right)$.
  \end{enumerate}
\end{theorem}

\subsection{The odd part}
We rewrite the odd part of the Wigner equation, i.e., the IVP
\eqref{WignerVecOdd}+\eqref{WignerVecOddIV}, into
\begin{equation} \label{WVO1}
  \opd{f(x)}{x} -B(x) f(x) = 0,
\end{equation}
with the initial condition 
\begin{equation} \label{WVOI1}
  f(-l/2) = f_b \in L^2_o(\Rv).
\end{equation}

We instantly declare that the solution of the IVP 
\eqref{WVO1}-\eqref{WVOI1} has to be an odd function.
\begin{lemma} \label{MustOdd}
  If $f_b(v) \in L^2_o(\Rv)$ and $f(x) \in L^2(\Rv)$ is the solution of 
 the IVP \eqref{WVO1}-\eqref{WVOI1}, then $f(x) \in L^2_o(\Rv)$. 
\end{lemma}
\begin{proof} 
  Let
  \[
   [g(x)](v) = \frac{[f(x)](v)+[f(x)](-v) }{2},
  \]
  where $f(x) \in L^2(\Rv)$ is the solution of the IVP
  \eqref{WVO1}-\eqref{WVOI1}. One may directly verify that $g(x) \in
  L^2_e(\Rv)$ and $g(x)$ is the solution of the IVP \eqref{WVE1} with
  zero initial value. By Theorem \ref{EvenWell}, we conclude that
  $g(x) = 0$, thus $f(x) \in L^2_o(\Rv)$.
\end{proof}

However, whether there exists a solution $f(x) \in L_o^2(\Rv)$ for
\eqref{WVO1}-\eqref{WVOI1} is difficult to discuss. A necessary 
condition for the existence is derived as follow.  We rewrite the 
Wigner equation into the form
\begin{equation}  \label{OddFA}
 \opd{f(x)}{x} -A(x) f(x) = 
  \frac{1}{v} \int_{\mathbb{R}_v'}
  \Vw \left( x,  v' \right) [f(x)](v') d v'
\end{equation}
where $A(x)$ defined in \eqref{OperatorA} has been proved to be 
a bounded linear operator on $L^2(\Rv)$.
For a function $f(x)\in  L_o^2(\Rv) $ with $\opd{f(x)}{x} \in  
L_o^2(\Rv)$, we know that the left hand side of \eqref{OddFA} is 
in $ L^2(\Rv)$ by using the boundedness of the operator $A(x)$, but 
the right hand side of \eqref{OddFA} is not in $ L^2(\Rv)$ unless 
the solution $f(x,v)$ satisfies 
\begin{equation} \label{ConditionOrthogonal}
 \int_{\mathbb{R}_v}
 \Vw \left( x,  v \right) [f(x)](v) d v =0.
\end{equation}
 
However, it is difficult  to give a condition for the initial value 
$f_b(v) \in L^2_o(\Rv)$ to ensure that there exists a solution $f(x)
\in L^2_o(\Rv)$ which satisfies the condition 
\eqref{ConditionOrthogonal}. So we will assume the existence, and 
discuss the uniqueness from the viewpoint of moments of the
distribution function.

Let us rewrite the Wigner equation \eqref{WVO1} into 
\begin{equation} \label{wfxv}
  v \pd{f(x,v)}{x} - \int \Vw(x,v-v') f(x,v') \dd v' = 0.  
\end{equation}
For $n \in \bbN^+ :=\{0,1,2,\cdots,\}$, we define 
\begin{equation} \label{VwnDef}
  J_n(x) =\int_{\Rv} v^n f(x,v) \dd v, \quad 
  \Vw_n(x) = \int_{\Rv} v^n \Vw(x,v) \dd v. 
\end{equation}
If the moment generating function of $f(x,v)$ 
$$
M_{v}[f](x,t) = \int_{-\infty}^{\infty} e^{v t} f(x,v) \dd v 
= \sum_{n=0}^{\infty} \frac{t^n}{n!} J_n(x) 
$$
exists for an open interval containing $t=0$, then $f(x,v)$ can be
represented into the bilateral Laplace transform of $M_{v}(f)$, i.e.,
\begin{equation}
f(x,v) = \int_{-\infty}^{\infty}  e^{-vt} M_{v}[f](x,t) \dd t,
\end{equation}
which implies $f(x,v)$ is completely determined by all its
moments. 

Recalling that $f(x,v)$ is an odd function of $v$ according to 
Lemma \ref{MustOdd}, we have that 
\begin{equation} \label{JnODE}
  J_n(x) = 0, \quad n = 0,2,4,\cdots. 
\end{equation}
Noticing that $\Vw(x,v)$ is an odd function of $v$, we integrate 
\eqref{wfxv} with respect to $v$ and obtain 
\[
  \opd{ J_1(x)}{x} = 0.
\]
Multiplying $v^2$ on both sides of \eqref{wfxv}, a simple calculation yields 
\begin{equation}
  \opd{J_3(x)}{x} - 2 J_1(x) \Vw_1(x) = 0.  
\end{equation}
Similarly we can obtain the differential equations for $n=5,7,\cdots$. 
Generally, we can write out the differential equations for $J_n(x)$,  
\begin{equation} \label{JnODE1}
  \opd{J_n}{x} - \sum_{k=1,3,\cdots,n-2}\binom{n-1}{k} 
  \Vw_k(x) J_{n-1-k}(x)   = 0, \quad  n = 1,3,5,\cdots,
\end{equation} 
where for $m\in \bbN^+, n \in \bbN^+$,   
\[
  \binom{m}{n} = \begin{cases}
	\frac{m!}{n!(m-n)!}, & m \geq n, \\
					 0 , & m < n .  
  \end{cases}
\]

Using \eqref{JnODE1} and \eqref{JnODE}, we obtain that  if the 
initial value $f_b = 0$, then 
\begin{equation} \label{Jn0}
  J_n(x) = 0 , \quad \forall x \in \bbR, \quad \forall n \in \bbN^+. 
\end{equation}
If $f_b=0$ and $f(x,\cdot) \in L^2(\Rv)$, then $f(x,v) = 0
$. Otherwise, there exists an $x\in(-l/2,l/2)$ such that $f(x,v) \neq
0$, which implies
\begin{equation} \label{contradict}
  \int_{\Rv} |f(x,v) |^2 \dd v \neq 0.  
\end{equation}
From \eqref{Jn0}, we have
$J_n(x)=0$ for all $n\in \bbN^+$.  Since $f(x, \cdot)\in L^2(\Rv)$,
$f(x, v)$ can be approximated by polynomial sequence $P_n(v)$ that $
\| f(x, \cdot) - P_n(\cdot) \|_{L^2(\Rv)} \rightarrow 0$ as
$n\rightarrow \infty$, and
\begin{equation}
  \int P_n(v) f(x,v) \dd v = 0. 
\end{equation}
Taking the limit as $n\rightarrow \infty$, we have 
\begin{equation}
  \int | f(x,v) |^2 \dd v = 0, 
\end{equation}
which contradicts with \eqref{contradict}. This gives $f(x,v)=0$ for 
all $x\in(-l/2,l/2)$, which gives us the uniqueness that
\begin{theorem}
  If $f_b \in L^2_o(\Rv)$, $f_1(x)$ and $f_2(x) \in L_o^2(\Rv)$ are
  two solutions of the IVP \eqref{WVO1}-\eqref{WVOI1}, then
  $f_1(x,v)=f_2(x,v)$.
\end{theorem}

%


\section{Discussion on stationary Inflow Boundary Value Problem} 
With the results on the initial value problems, we are ready to study
the stationary Wigner equation with inflow boundary conditions, i.e.,
\begin{equation} \label{WignerBVP}
  \opd{}{x}f(x) + B(x) f(x) = 0, \quad x\in(-l/2,l/2),
\end{equation}
and
\begin{equation} \label{WignerBVPBC}
  P^+ f(-l/2) = f_L, \quad P^-f(l/2) = f_R,  
\end{equation}
where $f_L \in L^2(\Rv^+)$ and $f_R \in L^2(\Rv^-)$.  Here $P^{\pm}:
L^2(\Rv)\rightarrow L^2(\Rv^{\pm})$ are defined by
\[
  [P^+ u ] (v) = u(v) ,  \text{ if } v \in \Rv^+, 
  \quad 
  [P^- u ] (v) = u(v) ,   \text{ if } v \in \Rv^-,  
\]
where $\Rv^+ = \{  v>0\}$ and $\Rv^- = \{v<0\}$. 

Let us assume that there is a solution in $L^2(\Rv)$ for the
BVP
\eqref{WignerBVP}-\eqref{WignerBVPBC}. Due to the parity decomposition
of the solution, the odd part of the solution has to satisfy the
equations \eqref{JnODE1}. The equations \eqref{JnODE1} actually
give a one-to-one linear mapping between the odd part of the solution
at the boundarys, which is denoted as
\[
  [P_o f(l/2)](v) = Q_{l\rightarrow r} [P_o f(-l/2)](v), \quad
  [P_o f(-l/2)](v) = Q_{r\rightarrow l} [P_o f(l/2)](v).
\]
Here $Q_{l\rightarrow r}$ is the map of the odd part of the solution
from the left end to the right end, and $Q_{r\rightarrow l} =
Q_{l\rightarrow r}^{-1}$ is the inverse mapping of $Q_{l\rightarrow
  r}$. Here we point out that actually $Q_{r\rightarrow l}$ is given
by the solution of system
\[
  \opd{J_n}{x} + \sum_{k=1,3,\cdots,n-2}\binom{n-1}{k} 
  \Vw_k(x) J_{n-1-k}(x)   = 0, \quad  n = 1,3,5,\cdots
\]

Meanwhile, by theorem \ref{EvenWell}, there is a one-to-one linear
mapping between the even part of the solution at the boundaries,
too. We denote this map as
\[
  [P_e f(l/2)](v) = R_{l\rightarrow r} [P_e f(-l/2)](v), \quad
  [P_e f(-l/2)](v) = R_{r\rightarrow l} [P_e f(l/2)](v).
\]
Here $R_{l\rightarrow r}$ is the map of the even part of the solution
from the left end to the right end, and $R_{r\rightarrow l} =
R_{l\rightarrow r}^{-1}$ is the inverse mapping of $R_{l\rightarrow
  r}$.

Then we have the relations that for $v < 0$,
\[
\begin{array}{rcl}
f_R(v) &=& [P_o f(l/2)](v) + [P_e f(l/2)](v) \\ [2mm]
       &=& Q_{l\rightarrow r} [P_o f(-l/2)](v) + R_{l\rightarrow r} [P_e f(-l/2)](v) \\ [2mm]
       &=& Q_{l\rightarrow r} [-P_o f(-l/2)](-v) + R_{l\rightarrow r} [P_e f(-l/2)](-v) \\ [2mm]
       &=& -Q_{l\rightarrow r} [P_o f(-l/2)](-v) + R_{l\rightarrow r} [P_e f(-l/2)](-v), \\ [2mm]
f_L(-v) &=& [P_o f(-l/2)](-v) + [P_e f(-l/2)](-v).
\end{array}
\]
We can solve $[P_o f(-l/2)](-v) $,$ [P_e f(-l/2)](-v)$ from the equation:
\begin{equation}\label{eq:P_e f_L}
 [P_e f(-l/2)](v) =
 \left\{
 \begin{array}{ll}
 (Q_{r\rightarrow l} R_{l\rightarrow r} + I)^{-1}(Q_{r\rightarrow
l}f_R(-v)+f_L(v)), & v>0 \\ [2mm]
 [ P_e f(-l/2) ](-v), & v<0 
 \end{array}
 \right.
\end{equation}
and 
\begin{equation}\label{eq:P_o f_L}
 [P_o f(-l/2)](v) =
 \left\{
 \begin{array}{ll}
 f_L(v)-[P_e f(-l/2)](v), & v>0 \\ [2mm]
 -[ P_o f(-l/2) ](-v), & v<0 
 \end{array}
 \right.
\end{equation}

Thus the solution $[f(x)](v)$ of the BVP
\eqref{WignerBVP}-\eqref{WignerBVPBC} 
can be solved, and it can be decomposed into
the sum of $f_o(x) \in L^2_o(\Rv)$ and $f_e(x) \in L_e^2(\Rv)$, 
\begin{equation}
  f(x)=f_o(x)+f_e(x) 
\end{equation}
where $f_o(x)$ is obtained in the sense that all its moments can be 
obtained by solving the ODEs \eqref{JnODE1} with the initial value 
obtained through \eqref{eq:P_o f_L} , and $f_e(x)$ can be obtained by solving 
\eqref{WVE1} with the initial value given by \eqref{eq:P_e f_L}.

Particularly, for a simple case that $V(x)$ is an even
function, i.e, $V(-x)=V(x)$ (for example, $V(x) = \exp(-x^2/a)$ where
$a>0$ is a constant). In this case, we have that $\Vw_n(x)$ is odd,
\begin{equation} \label{EvenVw}
\Vw_n(-x) = -\Vw_n(x) 
\end{equation}
which can be verified by using \eqref{VwnDef} and \eqref{VwDef}.    

Observing \eqref{JnODE1} and using \eqref{EvenVw}, we can derive that 
$J_n(x)$ is even, i.e., 
\begin{equation}
J_n(-x) = J(x). 
\end{equation}
Especially,  
\begin{equation} \label{JnBothSides}
J_n(-l/2) = J_n(l/2) 
\end{equation}
which means 
\begin{equation}
 f_o(-l/2)=f_o(l/2).
\end{equation}
That is to say $ Q_{l\rightarrow r}=Q_{r\rightarrow l}=I$. Using
the symmetry analysis in \cite{LiLuSun2014}, we can show that
$R_{l\rightarrow r}=R_{r\rightarrow l} = I$. Then using \eqref{eq:P_e
f_L} and \eqref{eq:P_o f_L}, we obtain the initial values for the even
part and the odd part of the Wigner equation. Finally, the solution
of BVP \eqref{WignerBVP}-\eqref{WignerBVPBC} is constructed by the
solutions of the two IVPs.


\section{Conclusion}
We studied the Wigner equation with inflow boundary conditions
by parity decomposition. The pseudo-operation $\Theta[V]$ is proved
to be bounded for the even $L^2$-space, so the propagator
for the even Wigner IVP is invertible. For the odd part of the Wigner
function whose moment generating function exists, we can calculate
the Wigner function through calculating its moments. With the help of 
analysis in parity decomposition, we plan to design an implementable
moment method for the Wigner equation.  

\section*{Acknowledgements}
This research was supported in part by NSFC (91230107, 11325102,
91434201). 


\end{document}